\documentclass[11pt]{amsart}
\usepackage{mathtools}
\usepackage{amssymb}
\usepackage[margin=1in]{geometry}
\usepackage[utf8]{inputenc}
\usepackage[pdfdisplaydoctitle=true,
            colorlinks=true,
            urlcolor=blue,
            citecolor=blue,
            linkcolor=blue,
            pdfstartview=FitH,
            pdfpagemode= UseNone,
            bookmarksnumbered=true]{hyperref}
\usepackage{float} %allows for arbitrary table placement
\usepackage{graphicx}
\usepackage{epstopdf}
\newtheorem{theorem}{Theorem}
\newtheorem{corollary}[theorem]{Corollary}

\newtheorem{proposition}[theorem]{Proposition}
\theoremstyle{definition}

\theoremstyle{remark}
\newtheorem{remark}[theorem]{Remark}
\newcommand{\PSL}{\mathrm{PSL}}

\newcommand{\abs}[1]{\left\vert#1\right\vert}

\DeclareMathOperator{\sgn}{sign} %

\newcommand{\id}{\mathrm{id}}
\newcommand{\Rot}{\mathrm{Rot}}
\newcommand{\Diff}{\mathrm{Diff}}

\mathtoolsset{showonlyrefs}
\hypersetup{
    pdftitle={Global smoothness for the Euler-Weil-Petersson equation}, %  Title
    pdfauthor={Stephen C. Preston}, % Author
    pdfsubject={MSC 2010: 35Q35, 53D25}, % Subject
    pdfkeywords={Euler-Weil-Petersson equation, global existence, universal Teichm\"uller space, Sobolev metrics of fractional order}, % Keywords
    pdflang=EN % Language
    }
\begin{document}

\title{Global existence and blowup for geodesics in universal Teichm\"uller spaces}%

\author[S. C. Preston]{Stephen C. Preston}
\address{Department of Mathematics, Brooklyn College, Brooklyn, NY 11210, USA}
\email{stephen.preston@brooklyn.cuny.edu}

\author[P. Washabaugh]{Pearce Washabaugh}
\address{Department of Mathematics, University of Colorado, Boulder, CO 80309-0395, USA}
\email{pearce.washabaugh@colorado.edu}

\thanks{S. C. Preston gratefully acknowledges support from Simons Collaboration Grant no. 318969, and the hospitality of the organizers of the ``Math on the Rocks'' workshop in Grundsund, Sweden in July 2015, where part of this research was conducted.}%

\date{\today}%
%\dedicatory{}%

\begin{abstract}
In this paper we prove that all initially-smooth solutions of the Euler-Weil-Petersson equation, which describes geodesics on the universal Teichm\"uller space under the Weil-Petersson metric, will remain smooth for all time. This extends the work of Escher-Kolev for strong Riemannian metrics to the borderline case of $H^{3/2}$ metrics. In addition we show that all initially-smooth solutions of the Wunsch equation, which describes geodesics on the universal Teichm\"uller curve under the Velling-Kirillov metric, must blow up in finite time due to wave breaking, extending work of Castro-C\'ordoba and Bauer-Kolev-Preston. Finally we illustrate these phenomena in terms of conformal maps of the unit disc, using the conformal welding representation of circle diffeomorphisms which is natural in Teichm\"uller theory.
\end{abstract}

\maketitle

\tableofcontents

\section{Introduction}

In this paper we are interested in smooth solutions to the one-dimensional Euler-Arnold equation
\begin{equation}\label{eulerarnold}
m_t + um_{\theta} + 2m u_{\theta} = 0, \qquad m = \Lambda u, \qquad u = u(t,\theta), \qquad u(0) = u_0\in C^{\infty}(S^1),
\end{equation}
where $\Lambda$ is a symmetric, positive semi-definite nonlocal differential operator and $S^1=\mathbb{R}/2\pi \mathbb{Z}$.
The particular special cases we are interested in are
\begin{itemize}
\item the Euler-Weil-Petersson equation~\cite{GBR2015}: $\Lambda = -H(u_{\theta\theta\theta}+u_{\theta})$ or $\Lambda(e^{in\theta}) = \lvert n\rvert (n^2-1) e^{in\theta}$,
\item and the Wunsch equation~\cite{Wun2010},\cite{BKP2015}: $\Lambda = Hu_{\theta}$ or $\Lambda(e^{in\theta}) = \lvert n\rvert e^{in\theta}$,
\end{itemize}
where $H$ is the Hilbert transform defined by $H(e^{in\theta}) = -i(\sgn{n}) e^{in\theta}$.
When paired with the flow equation
\begin{equation}\label{flow}
\frac{\partial\eta}{\partial t}(t,\theta) = u\big(t, \eta(t,\theta)\big), \qquad \eta(0,\theta)=\theta,
\end{equation}
the Euler-Arnold equation \eqref{eulerarnold} describes geodesics $\eta(t)$ of the right-invariant
Riemannian metric defined at the identity element by
$\langle u, u\rangle = \int_{S^1} u\Lambda u\, d\theta$ on a homogeneous space
$\Diff(S^1)/G$. Here $G$ is the group generated by the subalgebra $\ker{\Lambda}$ of length-zero directions: for the Euler-Weil-Petersson
equation we have $G = \PSL_2(\mathbb{R})$, and for the Wunsch equation we have $G = \Rot(S^1) \cong S^1$. See Arnold-Khesin~\cite{AK1998} for the general
setup (which includes also the Euler equations of ideal fluid mechanics), Khesin-Misio{\l}ek~\cite{KM2003} for the setup for homogeneous spaces,
and Escher-Kolev~\cite{EK2014a} for these two equations in particular.

The local existence result is that if $u_0\in H^s(S^1)/\mathfrak{g}$ for $s>\tfrac{3}{2}$ (where $\mathfrak{g}$ is the Lie algebra of $G$), then there is a unique solution $u\in C([0,T), H^s(S^1)/\mathfrak{g})$ for some
$T>0$ (which may be infinite). In our context this is a consequence of the fact that the geodesic equation is smooth, so that there is a unique solution $\eta\in C^{\infty}([0,T), \Diff^s(S^1)/G)$ with $\eta(0)=\id$ and $\dot{\eta}(0)=u_0$. Loss of smoothness of $u$ in time occurs due to the fact that composition required to get $u=\dot{\eta}\circ\eta^{-1}$ is not smooth. This approach to the Euler equations was originally due to Ebin-Marsden~\cite{EM1970}; for the Wunsch equation it was proved by Escher-Kolev-Wunsch~\cite{EKW2012}, while for the Euler-Weil-Petersson equation it was proved by Escher-Kolev~\cite{EK2014a}. Castro-C\'ordoba~\cite{CC2010} showed that if $u_0$ is initially odd, then solutions to the Wunsch equation blow up in finite time; the authors of \cite{BKP2015} generalized this to show that if there exists any $\theta_0\in S^1$ with $u_0'(\theta_0)<0$ and $Hu_0'(\theta_0)=0$, then the solution of the Wunsch equation blows up in finite time. For the Euler-Weil-Petersson equation, it was not known whether initially smooth data would remain smooth for all time. However Gay-Balmaz and Ratiu~\cite{GBR2015} interpreted the equation in $H^{3/2}$ as describing geodesics of a \emph{strong} Riemannian metric on a certain manifold and concluded that the velocity field $u$ remains in $H^{3/2}(S^1)$ for all time. This is almost but not quite enough to get a global $C^1$ bound, which would guarantee global existence of \emph{smooth} solutions along with a smooth Lagrangian flow $\eta$.

In general the Hilbert manifold topology on $\Diff^s(S^1)$ or related spaces (e.g., subgroups or quotients) is given in terms of the Sobolev space norms $H^s$ for $s>\frac{3}{2}$, since $u\in H^s(S^1)$ for $s>\frac{3}{2}$ ensures $u\in C^1(S^1)$. On the other hand the right-invariant Riemannian metric which gives the geometry and the geodesic equation is given in terms of the Sobolev $H^r$ norm for some $r\le s$. If $r>\frac{3}{2}$ then we can use $r=s$, and we have a strong Riemannian metric; global existence of geodesics in that case was proved by Escher-Kolev~\cite{EK2014b}. If $r\le 1$ then solutions may blow up; for example the $H^1$ metric on $\Diff(S^1)$ leads to the Camassa-Holm equation~\cite{Mis1998} which, for some initial data, has solutions that blow up in finite time~\cite{McK2003}. The degenerate $\dot{H}^1$ metric on $\Diff(S^1)/\Rot(S^1)$ leads to the Hunter-Saxton equation~\cite{KM2003}, for which all solutions blow up in finite time~\cite{Len2007}.

The main theorems of this paper settle the global existence question for the degenerate $\dot{H}^r$ metrics corresponding to $r=\frac{1}{2}$ (the Wunsch equation) and $r=\frac{3}{2}$ (the Euler-Weil-Petersson equation).

\begin{theorem}\label{mainthm1}
Suppose $s>\tfrac{3}{2}$ and $u_0$ is an $H^s$ velocity field on $S^1$ with mean zero (i.e., $u_0\in H^s(S^1)/\mathbb{R}$). Then the solution $u(t)$ of the Wunsch equation with $u(0)=u_0$ blows up in finite time.
\end{theorem}

\begin{theorem}\label{mainthm2}
Suppose $s>\tfrac{3}{2}$ and $u_0$ is an $H^s$ velocity field on $S^1$, and that the Fourier series of $u_0$ has vanishing $n=0$, $n=1$, and $n=-1$ component; i.e., $u_0\in H^s(S^1)/\mathfrak{sl}_2(\mathbb{R})$. Then the solution $u(t)$ of the Euler-Weil-Petersson equation with $u(0)=u_0$ remains in $H^s$ for all time. In particular if $u_0$ is $C^{\infty}$ then so is $u(t)$ for all $t>0$.
\end{theorem}

Additionally, Theorem \ref{mainthm1} almost immediately gives us that every mean zero solution of the Constantin-Lax-Majda equation  \cite{CLM1985} blows up in finite time. Overall, these two Theorems mean that the case $r=\tfrac{3}{2}$ behaves the same as the cases for $r>\frac{3}{2}$, while the case $r=\frac{1}{2}$ behaves the same as for $r=1$. We may conjecture that there is a critical value $r_0$ such that for $r>r_0$ all smooth mean-zero solutions remain smooth for all time, while for $r<r_0$ all smooth mean-zero solutions blow up in finite time. Our guess is that $r_0=\frac{3}{2}$, but the current method does not prove this; furthermore we do not know what happens with geodesics for $\frac{1}{2}<r<1$ or $1<r<\frac{3}{2}$ even in the degenerate case.

%In addition we will present some of the background for both equations as they arise naturally in Teichm\"uller theory.
Both equations arise naturally in the study of universal Teichm\"uller spaces. The Euler-Weil-Petersson equation was derived in \cite{GBR2015} as the Euler-Arnold equation arising from the Weil-Petersson metric on the universal Teichm\"uller space. This geometry has been studied extensively by Takhtajan-Teo~\cite{TT2006}; in particular they constructed the Hilbert manifold structure that makes Weil-Petersson a strong Hilbert metric (thus ensuring that geodesics exist globally). The Weil-Petersson geometry is well-known: the sectional curvature is strictly negative, and it is a K\"ahler manifold with almost complex structure given by the Hilbert transform. See Tromba~\cite{Tro1992} and Yamada~\cite{Yam2014} for further background on the Weil-Petersson metric on the universal Teichm\"uller space.

The Wunsch equation arises from the Riemannian metric $\langle u, u\rangle = \int_{S^1} u Hu_{\theta}\,d\theta$, which is called the Velling-Kirillov metric and was proposed as a metric on the universal Teichm\"uller curve by Teo~\cite{Teo2004}\cite{Teo2007}. The Velling-Kirillov geometry was originally studied by Kirillov-Yur'ev~\cite{KY1987}; although the sectional curvature is believed to be always positive, this is not yet proved. Furthermore the geometries are related in the sense that integrating the square of the symplectic form for the W-P geometry gives the symplectic form for the V-K geometry. Yet the properties of these geometries seem to be opposite in virtually every way: from Fredholmness of the exponential map~\cite{MP2010}\cite{BKP2015} to the sectional curvature to the global properties of geodesics mentioned above.

The authors would like to thank Martin Bauer and Boris Kolev for suggesting the problem and useful discussions on the result. %The second author would like to thank the organizers and participants of the ``Math on the Rocks'' shape analysis workshop in Grundsund, Sweden during July 2015, where parts of the proofs of several theorems were discovered.

\section{Proof of the Main Theorems}

\subsection{Rewriting the Equations and Proof of Theorem \ref{mainthm1}}

Let us first sketch the blowup argument for the Wunsch equation from \cite{BKP2015}.
The Wunsch equation is given for mean-zero vector fields $u$ on $S^1$ (identified with functions) by the formula
\begin{equation}\label{wunsch}
\omega_t + u\omega_{\theta} + 2u_{\theta} \omega = 0, \qquad \omega = Hu_{\theta}.
\end{equation}
In terms of the Lagrangian flow $\eta$ given by \eqref{flow}, we may rewrite this as
$$\partial_t \omega\big(t,\eta(t,\theta)\big) + 2 \eta_{t\theta}(t,\theta) \omega\big(t,\eta(t,\theta)\big)/\eta_{\theta}(t,\theta) = 0$$
which leads to the conservation law
\begin{equation}\label{conservationlaw}
\eta_{\theta}(t,\theta)^2 \omega\big(t,\eta(t,\theta)\big) = \omega_0(\theta).
\end{equation}
Applying the Hilbert transform to both sides of \eqref{wunsch} and using the following Hilbert transform identities (valid for mean-zero functions $f$):
\begin{equation}\label{hilbertidentities}
H(Hf)=-f \qquad \text{and}\qquad 2H(fHf) = (Hf)^2 - f^2,
\end{equation}
one obtains~\cite{BKP2015} an equation for $u_{\theta} = -H\omega$:
\begin{equation}\label{uthetawunsch}
u_{t\theta} + uu_{\theta\theta} + u_{\theta}^2 = -F + \omega^2
\end{equation}
where the function $F$ is a spatially nonlocal force given for each fixed time $t$ by
\begin{equation}\label{Fdef}
F = -u u_{\theta\theta} - H(uHu_{\theta\theta}).
\end{equation}
For a mean-zero function $u$, it was shown in \cite{BKP2015} that $F(t,\theta)>0$ for every $t$ and $\theta\in S^1$. We will give an
alternate proof of this fact in Theorem \ref{magicformulathm} below.

In Lagrangian form, using the conservation law  equation \eqref{uthetawunsch} becomes
\begin{equation}\label{wunschlagrangian}
\eta_{tt\theta}(t,\theta) = \frac{\omega_0(\theta)^2}{\eta_{\theta}(t,\theta)^3} - F\big(t,\eta(t,\theta)\big)\eta_{\theta}(t,\theta).
\end{equation}
It follows that if there is a point $\theta_0$ such that $u_0'(\theta_0)< 0$ and $\omega_0(\theta_0)=0$, then we will have
$\eta_{\theta}(0,\theta_0)=1$, $\eta_{t\theta}(0,\theta_0)< 0$, and $\eta_{tt\theta}(t,\theta_0)<0$ for all $t$, so that $\eta_{\theta}(t,\theta_0)$ must reach zero in finite time (which leads to $u_{\theta} \to -\infty$). Our proof that \emph{all} solutions blow up consists of showing that this condition happens for \emph{every} initial condition $u_0$ with $\omega_0=Hu_0$.

\begin{proof}[Proof of Theorem \ref{mainthm1}]
From the discussion above, the proof reduces to proving the following statement. Suppose $f\colon S^1\to \mathbb{R}$ is a smooth function with mean zero, and let $g = Hf$. Then there is a point $\theta_0\in S^1$ with $f'(\theta_0)< 0$ and $g'(\theta_0)=0$.

Let $p$ be the unique harmonic function in the unit disc $\mathbb{D}$ such that $p\vert_{S^1} = f$, and let $q$ be its harmonic conjugate normalized so that $q\vert_{S^1} = g$. Then in polar coordinates we have the Cauchy-Riemann equations
\begin{equation}\label{cauchyriemann}
r p_r(r,\theta) = q_{\theta}(r,\theta) \qquad \text{and}\qquad r q_r(r,\theta) = -p_{\theta}(r,\theta),
\end{equation}
and we have $p(1,\theta) = f(\theta)$ and $q(1,\theta) = g(\theta)$.

Since $q$ is harmonic, its maximum value within $\mathbb{D}$ occurs on the boundary $S^1$ at some point $\theta_0$. The maximum of $g$ occurs at the same point, so that $g'(\theta_0)=0$. By the Hopf lemma, we have $q_r(1,\theta_0) > 0$, so equations \eqref{cauchyriemann} imply that $f'(\theta_0) = p_{\theta}(1,\theta_0) < 0$.
\end{proof}

\begin{remark}
This argument also works when the domain is $\mathbb{R}$ and the functions have suitable decay conditions imposed. It can thus be applied to demonstrate that \emph{every} mean zero solution of the Constantin-Lax-Majda equation \cite{CLM1985}
$$\omega_t-v_x\omega = 0,\;\; v_x=H\omega$$
blows up in finite time, using the same argument as in that paper via the explicit solution formula.
\end{remark}

Now let us rewrite the Euler-Weil-Petersson equation to obtain the analogue of formula \eqref{uthetawunsch}. Recall from the introduction that it is given explicitly by
\begin{equation}\label{EWP}
\omega_t + u\omega_{\theta} + 2u_{\theta}\omega = 0, \qquad \omega = -Hu_{\theta\theta\theta} - Hu_{\theta}.
\end{equation}

\begin{proposition}\label{EWPrewrite}
The Euler-Weil-Petersson equation \eqref{EWP} is equivalent to the equation
\begin{equation}\label{EWPint}
u_{t\theta} = H(uHu_{\theta\theta}) + H(1+\partial_{\theta}^2)^{-1}\big[ 2u_{\theta} Hu_{\theta} - u_{\theta\theta} Hu_{\theta\theta}\big],
\end{equation}
In terms of the Lagrangian flow \eqref{flow}, equation \eqref{EWPint} takes the form
\begin{equation}\label{forcing}
\frac{\partial}{\partial t} u_{\theta}(t,\eta(t,\theta)) = -F(t,\eta(t,\theta)) + G(t,\eta(t,\theta))
\end{equation}
where $F$ is defined by formula \eqref{Fdef}
and $G$ is given by
\begin{equation}\label{Gdef}
G = H(1+\partial_{\theta}^2)^{-1}[2u_{\theta}Hu_{\theta} - u_{\theta\theta}Hu_{\theta\theta}].
\end{equation}
Here the operator $(1+\partial_{\theta}^2)$ is restricted to the orthogonal complement of the span of $\{1, \sin{\theta}, \cos{\theta}\}$ so as to be invertible.
\end{proposition}

\begin{proof}
Equation \eqref{EWP} may be written
%\begin{align*}
%0 &= H(1+\partial_{\theta}^2)u_{t\theta} + uHu_{\theta\theta\theta\theta} + uHu_{\theta\theta} + 2u_{\theta} Hu_{\theta\theta\theta} + 2u_{\theta} Hu_{\theta} \\
$$
-H(1+\partial_{\theta}^2)u_{t\theta} =  (1+\partial_{\theta}^2)(uHu_{\theta\theta}) - u_{\theta\theta}Hu_{\theta\theta} + 2u_{\theta}Hu_{\theta},
$$
%\end{align*}
using the product rule. We now solve for $u_{t\theta}$ by applying $H$ to both sides and inverting $(1+\partial_{\theta}^2)$.
%Now apply $-H$ to both sides using $H^2=-1$.
%To apply $(1+\partial_{\theta}^2)^{-1}$,
To do this, we just need to check that the term $(2u_{\theta}Hu_{\theta} - u_{\theta\theta}Hu_{\theta\theta})$
is orthogonal to the subspace spanned by $\{1, \sin{\theta}, \cos{\theta}\}$. In fact this is true for every function $fHf$ when $f$ is $2\pi$-periodic with mean zero, since the formulas \eqref{hilbertidentities} imply both that $fHf$ has mean zero and that it has period $\pi$.

%This follows from two facts. First, every function is orthogonal to its Hilbert transform since $\int_{S^1} f(\theta) Hf(\theta) \, d\theta = 0$. %Second, the identity implies that $2fHf$ has period $\pi$ for any $2\pi$-periodic function $f$, and thus it is a sum of terms of frequency at least $2$ (thus orthogonal to $\sin{\theta}$ and $\cos{\theta}$). %Finally equation \eqref{EWPint2} simply follows from using \eqref{magicproduct} in \eqref{EWPint}.

The only additional thing happening in equation \eqref{forcing} is the chain rule formula
$$ \partial_t u_{\theta}(t,\eta(t,\theta)) = u_{t\theta}(t,\eta(t,\theta)) + u_{\theta\theta}(t,\eta(t,\theta)) \eta_t(t,\theta) = (u_{t\theta} + uu_{\theta\theta})(t,\eta(t,\theta)).$$
\end{proof}

%\subsection{Global existence of EWP solutions}

To prove Theorem \ref{mainthm2}, we want to show that $\lVert u_{\theta}\rVert_{L^{\infty}}$ remains bounded for all time, and by formula \eqref{forcing} it is sufficient to bound both $\lVert F\rVert_{L^{\infty}}$ and $\lVert G\rVert_{L^{\infty}}$. We will do this in the next two sections.

\subsection{The Bound on $F$}

In \cite{BKP2015}, it was shown that the function $F$ given by \eqref{Fdef} is positive for any mean-zero function $u\colon S^1\to \mathbb{R}$. This is essential for proving blowup for the Wunsch equation.

\begin{theorem}[Bauer-Kolev-Preston]\label{magicinequalityold}
Let $u\colon S^1\to\mathbb{R}$ be a function with Fourier series $u(\theta) = \sum_{n\in\mathbb{Z}} c_n e^{in\theta}$ with $c_0=0$.
%Let $H$ denote the Hilbert transform given by $H(e^{in\theta}) = -i\sgn{n} e^{in\theta}$.
If $F = -uu'' - H(uHu'')$, then
%If $\Lambda = H\partial_{\theta}$ so that $\Lambda(e^{in\theta}) = \lvert n\rvert e^{in\theta}$, and if $g_p = H(uH\Lambda^pu) + u\Lambda^pu$ for a positive number $p$, then for every $\theta\in S^1$ we have
\begin{equation}\label{magic}
F(\theta) = 2\sum_{n=1}^{\infty} (2n-1) \lvert \phi_n(\theta)\rvert^2, \qquad \text{where } \phi_n(\theta) = \sum_{m=n}^{\infty} c_m e^{im\theta}.
 \end{equation}
In particular $F(\theta)>0$ for every $\theta$ if $u$ is not constant.
\end{theorem}

%In this theorem we may assume without loss of generality that the mean of $f$ is zero (i.e., that $c_0=0$) since otherwise we may write $f = c_0 + \tilde{f}$, with $\tilde{f}$ of mean zero) and immediately check that $g_p = \tilde{g}_p$, where $g_p$ corresponds to $f$ and $\tilde{g}_p$ corresponds to $\tilde{f}$. We will do this from now on since it makes several computations simpler later on.
In fact as \cite{BKP2015} show, this theorem generalizes to any number of derivatives: we have $H(uH\Lambda^pu) + u\Lambda^pu\ge 0$ for any $p\ge 0$, with Theorem \ref{magicinequalityold} representing the case $p=2$. Although the theorem is not difficult to prove using some simple index manipulations, it is rather mysterious why it works. The following more explicit formula for the function makes clear why $F$ is positive.

\begin{theorem}\label{magicformulathm}
Suppose $u\colon S^1\to \mathbb{R}$ has Fourier series $u(\theta) = \sum_{n\in\mathbb{Z}} c_n e^{in\theta}$ with $c_0=0$. Let $\Phi$ denote the holomorphic function in the unit disc $\mathbb{D}$ given by $\Phi(z) = \sum_{n=1}^{\infty} c_n z^n$, and set $\Upsilon(w,z) = \frac{\Phi(w)-\Phi(z)}{w-z}$ (which is holomorphic in $\mathbb{D}\times \mathbb{D}$). If $F = -uu'' - H(uHu'')$, then
\begin{equation}\label{Fdifferencequotient}
F(\theta) = \frac{1}{\pi} \int_0^{2\pi} \big\lvert \Upsilon(e^{i\theta}, e^{i\psi})\big\rvert^2 \, d\psi
+ \frac{4}{\pi} \iint_{\mathbb{D}} \Big\lvert \frac{\partial \Upsilon}{\partial w}(w, e^{i\theta})\Big\rvert^2 \, dA.
\end{equation}
\end{theorem}

\begin{proof}
For each $n\ge 1$, set
$ \Phi_n(z) = \sum_{m=n}^{\infty} c_m z^m$. %Since $c_m = \frac{\Phi^{(m)}(0)}{m!}$,
The Cauchy integral formula says that
$$ c_m = \frac{1}{2\pi i} \int_{\gamma} \frac{\Phi(w) \, dw}{w^{m+1}},$$
and thus we obtain
$$ \Phi_n(z) = \sum_{m=n}^{\infty} \frac{1}{2\pi i} \int_{\gamma} \frac{z^m}{w^{m+1}} \, \Phi(w)\,dw = \frac{1}{2\pi i} \int_{\gamma} \frac{z^n \Phi(w) \, dw}{w^n (w-z)}$$
for any curve $\gamma$ around the origin, in particular when $\gamma = S^1$.
Since $n\ge 1$ we know that $\int_{\gamma} \frac{z^n}{w^n (w-z)} \, dw = 0$, and we conclude that
$$ \Phi_n(z) = \frac{1}{2\pi i} \int_{\gamma} \frac{z^n}{w^n} \Upsilon(w,z) \, dw.$$
%where $\Upsilon(w,z) = \frac{G(w)-G(z)}{w-z}$.

Using formula \eqref{magic}, we then have for $z=e^{i\theta}$ the formula
\begin{align*}
F(\theta) &= 2\sum_{n=1}^{\infty} (2n-1) \lvert \Phi_n(z)\rvert^2 \\
&= \frac{1}{2\pi^2} \int_{\gamma}\int_{\gamma} \sum_{n=1}^{\infty} (2n-1) \frac{\lvert z\rvert^{2n}}{w^n\overline{v}^n} \Upsilon(w,z)\overline{\Upsilon(v,z)} \, d\overline{v}\,dw \\
&= \frac{1}{2\pi^2} \int_{\gamma} \overline{\Upsilon(v,z)} \int_{\gamma}
\left( \frac{2\lvert z\rvert^4/\overline{v}^2}{(w-\lvert z\rvert^2/\overline{v})^2} + \frac{\lvert z\rvert^2/\overline{v}}{w-\lvert z\rvert^2/\overline{v}}\right) \Upsilon(w,z) \, dw \, d\overline{v}\\
&= \frac{i}{\pi} \int_{\gamma} \overline{\Upsilon(v,z)} \left[ \frac{2\lvert z\rvert^4}{\overline{v}^2} \Upsilon_w\big( \lvert z\rvert^2/\overline{v}, z\big) + \frac{\lvert z\rvert^2}{\overline{v}} \Upsilon\big( \lvert z\rvert^2/\overline{v}, z\big)\right] \, d\overline{v},
\end{align*}
using the Cauchy integral formula on the analytic function $w\mapsto \Upsilon(w,z)$. Now plugging in $z=e^{i\theta}$ and $v=e^{i\psi}$ and computing explicitly, we obtain
\begin{equation}\label{Fmidstep}
F(\theta) = \frac{1}{\pi} \int_0^{2\pi} \lvert \Upsilon(e^{i\psi}, e^{i\theta})\rvert^2 \, d\psi - \frac{2i}{\pi} \int_0^{2\pi} \overline{\Upsilon(e^{i\psi},e^{i\theta})} \partial_{\psi} \Upsilon(e^{i\psi},e^{i\theta})\, d\psi.
\end{equation}
It is easy to see that for any holomorphic function $Q$ on $\mathbb{D}$, we have
$$ \int_0^{2\pi} \overline{Q(e^{i\psi})} \partial_{\psi} Q(e^{i\psi}) \, d\psi = 2i \iint_{\mathbb{D}} \lvert Q'(w)\rvert^2 \, dA,$$
and formula \eqref{Fdifferencequotient} follows.
%To evaluate this last term, write (for each fixed $z=e^{i\theta}$), the real and imaginary parts of $\Upsilon(e^{i\psi}, e^{i\theta}) = a(\psi) + ib(\psi)$; by definition of the Hilbert transform $H$, we have $b = Ha$. Then the last term in equation \eqref{Fmidstep} becomes
%\begin{equation}\label{Falmostdone}
%- \frac{2i}{\pi} \int_0^{2\pi} \overline{\Upsilon(e^{i\psi},e^{i\theta})} \partial_{\psi} \Upsilon(e^{i\psi},e^{i\theta})\, d\psi = \frac{4}{\pi} %\int_0^{2\pi} a(\psi) Ha'(\psi)\,d\psi.
%\end{equation}
%Now in general if we have an analytic function $Q(r,\theta) = U(r,\theta) + iV(r,\theta)$ with $U(1,\theta)=u(\theta)$ and $V(1,\theta)=v(\theta)$, then it is easy to compute using the Cauchy-Riemann equations in polar coordinates $rU_r = V_{\theta}$ and $rV_r = -U_{\theta}$ that
%$$ \iint_{\mathbb{D}} \lvert Q'(w)\rvert^2 \, dA = \tfrac{1}{2} \iint_{\mathbb{D}} \lvert \grad U\rvert^2 + \lvert \grad V\rvert^2 \, dA
%= \tfrac{1}{2} \int_0^{2\pi} \int_0^1 \partial_r( rUU_r + rVV_r) \, dr\,d\theta = \int_0^{2\pi} u(\theta)v'(\theta).$$
%Plugging this into \eqref{Falmostdone} and then into \eqref{Fmidstep}, we obtain \eqref{Fdifferencequotient} as desired.
\end{proof}

We would now like to bound $F$ in terms of $\lVert u\rVert^2_{\dot{H}^{3/2}} := \int_{S^1} (Hu)(u'''+u') \, d\theta$. To do this we use Hardy's inequality to replace difference quotients with derivatives, then use the trace theorem to replace derivatives on the disc with half-derivatives on the boundary circle.

\begin{theorem}\label{mysteryboundthm}
Let $u\colon S^1\to \mathbb{R}$ be a smooth function, and let $F = -uu'' - H(uHu'')$. Then for every $\theta\in S^1$, we have
\begin{equation}\label{Finequality}
F(\theta) \le \tfrac{1}{\pi} \lVert u\rVert_{\dot{H}^{3/2}}^2 + \tfrac{\pi}{2} \lVert u\rVert_{\dot{H}^1}^2,
\end{equation}
where $\lVert u\rVert_{\dot{H}^{3/2}(S^1)}^2 = \int_{S^1} (Hu)(u'''+u') \, d\theta$ and $\lVert u\rVert^2_{\dot{H}^1(S^1)} = \int_{S^1} (u')^2 \, d\theta$.
\end{theorem}

\begin{proof}
First we show the following Hardy-type inequality~\cite{HLP1934}: for any $z\in S^1$ and analytic function $\Phi$ on $\mathbb{D}$, if $\Upsilon(w,z) = \frac{\Phi(w)-\Phi(z)}{w-z}$, then
\begin{equation}\label{differenceinequalitydisc}
\iint_{\mathbb{D}} \left\lvert  \frac{\partial}{\partial w} \Upsilon(w,z) \right\rvert^2\, dA \le \iint_{\mathbb{D}} \lvert \Phi''(w)\rvert^2 \, dA.
\end{equation}
Assume without loss of generality that $z=1$, and write $\Upsilon(w) := \Upsilon(w,1)$ for simplicity.
We have $\Upsilon(w) = \int_0^1 \Phi'(tw+1-t) \, dt$, so that $\partial_w \Upsilon(w,z) = \int_0^1 t \Phi''(tw + 1-t) \, dt$.
Then we have by Cauchy-Schwarz that
\begin{equation}\label{cauchyschwarz}
 \iint_{\mathbb{D}} \lvert \partial_w \Upsilon(w)\rvert^2 \, dA \le \int_0^{2\pi} \int_0^1 \int_0^1 t^2 r \big\lvert \Phi''(tre^{i\theta}+1-t)\big\rvert^2 \ dt \, dr \, d\theta.
 \end{equation}
 Define new variables $x=tr\cos{\theta}+1-t$, $y=tr\sin{\theta}$, and $s=t$. This transformation maps the unit cylinder $0\le r\le 1$, $0\le \theta\le 2\pi$, $0\le t\le 1$ to the unit cylinder $x^2+y^2\le 1$, $0\le s\le 1$. The Jacobian transformation is $dx\,dy\,ds = t^2r \, dt\,dr\,d\theta$, so that the inequality \eqref{cauchyschwarz} becomes
$$ \iint_{\mathbb{D}} \lvert \partial_w \Upsilon(w,z)\rvert^2 \, dA \le \int_{\mathbb{D}} \int_0^1 \lvert \Phi''(x+iy)\rvert^2 \, ds \, dx\,dy$$
which evaluates to \eqref{differenceinequalitydisc} immediately.

Furthermore for any fixed $\theta\in S^1$, we will show
\begin{equation}\label{differenceinequalitycircle}
\int_{S^1} \lvert \Upsilon(e^{i\psi}, e^{i\theta}) \rvert^2 \, d\psi \le \pi^2 \int_0^{2\pi} \lvert \Phi'(e^{i\psi})\rvert^2 \, d\psi.
 \end{equation}
Again without loss of generality we can assume $\theta=0$ and that $\Phi(1)=0$.
Define $\phi(\theta) = \Phi(e^{i\theta})$; then we want to show that
\begin{equation}\label{reduceddifferencecircle}
\int_{-\pi}^{\pi} \left\lvert \frac{\phi(\psi)}{e^{i\psi}-1}\right\rvert^2 \, d\psi \le \pi^2 \int_{-\pi}^{\pi} \lvert \phi'(\psi)\rvert^2 \, d\psi.
\end{equation}
First we observe that $$\lvert e^{i\psi}-1\rvert^2 = 4\sin^2{\psi/2} \ge \frac{4\psi^2}{\pi^2},$$
so that equation \eqref{reduceddifferencecircle} follows from the standard Hardy inequality~\cite{HLP1934} on the interval:
\begin{equation}\label{almostlaststep}
\int_{-\pi}^{\pi} \frac{\lvert \phi(\psi)\rvert^2}{\psi^2} \, d\psi \le 4 \int_{-\pi}^{\pi} \lvert \phi'(\psi)\rvert^2 \, d\psi.
\end{equation}

%Now we expand, using integration by parts, to obtain
%$$
%\int_{-\pi}^{\pi} \left\lvert \frac{g(\psi)}{2\psi} - g'(\psi)\right\rvert^2 \, d\psi = -\frac{1}{4} \int_{-\pi}^{\pi} \frac{\lvert %g(\psi)\rvert^2}{\psi^2}\,d\psi - \frac{\lvert g(\pi)\rvert^2}{\pi} + \int_{-\pi}^{\pi} \lvert g'(\psi)\rvert^2 \, d\psi.$$
%Since the left side is nonnegative, we obtain after rearrangement the inequality \eqref{almostlaststep}.

%Write $u(\theta) = \sum_{n\in\mathbb{Z}} c_n e^{in\theta}$ and set $\Phi(z) = \sum_{n=1}^{\infty} c_n z^n$, as in Theorem \ref{magicformulathm}.
Now formula \eqref{Fdifferencequotient} combined with \eqref{differenceinequalitydisc} and \eqref{differenceinequalitycircle} gives
$$ F(\theta) \le \pi \int_0^{2\pi} \lvert \Phi'(e^{i\psi})\rvert^2 \, d\psi
+ \frac{4}{\pi} \iint_{\mathbb{D}} \lvert \Phi''(w)\rvert^2 \, dA.$$
Now express this in terms of the Fourier coefficients $c_n$, using $\Phi(w) = \sum_{n=1}^{\infty} c_n w^n$, and a straightforward computation yields
$$ F(\theta) \le 2\pi^2 \sum_{n=1}^{\infty} n^2 \lvert c_n\rvert^2 +
4 \sum_{n=2}^{\infty} (n^3-n^2) \lvert c_n\rvert^2,$$
which leads directly to \eqref{Finequality}.
\end{proof}

\subsection{The Bound on $G$}

Note that $G$ given by \eqref{Gdef} consists of two similar terms, and the following Theorem takes care of both at the same time as a consequence of Hilbert's double series inequality.

\begin{theorem}\label{Gboundthm}
Suppose $f\colon S^1\to\mathbb{R}$ is a smooth function and that $g = H(1+\partial_{\theta}^2)^{-1}(f'Hf')$. Then $\lVert g\rVert_{L^{\infty}} \le 4\pi \lVert f\rVert^2_{\dot{H}^{1/2}}$.
\end{theorem}

\begin{proof}
Expand $f$ in a Fourier series as $f(\theta) = \sum_{n\in \mathbb{Z}} f_n e^{in\theta}$, and let $h = f'Hf'$. Then we have
\begin{align*}
f'Hf'(\theta) &= i \sum_{m,n\in\mathbb{Z}} mnf_m f_n (\sgn{n}) e^{i(m+n)\theta} = i \sum_{k\in \mathbb{Z}} \left(\sum_{n\in\mathbb{Z}} \abs{n} (k-n)f_{k-n} f_n  \right) e^{ik\theta} = i\sum_{k\in \mathbb{Z}} h_k e^{ik\theta},
\end{align*}
where
$$  h_k = \sum_{n\in\mathbb{Z}} \abs{n} (k-n) f_{k-n}f_n.$$
Now let us simplify $h_k$: we have for $k>0$ that
\begin{align*}
h_k &= \sum_{n=1}^{\infty} n(k-n) f_n f_{k-n} + \sum_{n=1}^{\infty} n(k+n) \overline{f_n} f_{k+n} \\
&= \sum_{n=1}^{k-1} n(k-n) f_n f_{k-n} + \sum_{m=1}^{\infty} (k+m)(-m) f_{k+m} \overline{f_m} + \sum_{n=1}^{\infty} n(k+n) \overline{f_n} f_{k+n},
\end{align*}
where we used the substitution $m=n-k$. Clearly the middle term cancels the last term, so
\begin{equation}\label{hkdef}
h_k = \sum_{n=1}^{k-1} n(k-n) f_n f_{k-n}.
\end{equation}
It is easy to see that $h_0=0$ due to cancellations, while if $k<0$, we get
$$ h_k = -\sum_{n=1}^{\abs{k}-1} n(\abs{k}-n) \overline{f_n} \overline{f_{\abs{k}-n}} = -\overline{h_{\abs{k}}}.$$
Note in particular that $h_1=h_{-1}=0$.
We thus obtain
$$ f'Hf'(\theta) = \sum_{k=2}^{\infty} \big( ih_k e^{ik\theta} - i\overline{h_k} e^{-ik\theta}\big),$$
so that
$$ H(f'Hf')(\theta) = \sum_{k=2}^{\infty} h_k e^{ik\theta} + \overline{h_k} e^{-ik\theta} = 2\text{Re}\left( \sum_{k=2}^{\infty} h_k e^{ik\theta}\right).$$

It now makes sense to apply $(1+\partial_{\theta}^2)^{-1}$ to this function, and we obtain
$$ g(\theta) = 2\text{Re}\left( \sum_{k=2}^{\infty} \frac{h_k}{1-k^2} e^{ik\theta}\right),$$ so that
\begin{align*}
\lVert g\rVert_{L^{\infty}} &\le 2 \sum_{k=2}^{\infty} \sum_{n=1}^{k-1} \frac{n(k-n) \lvert f_n\rvert \lvert f_{k-n}\rvert }{k^2-1} = 2\sum_{n=1}^{\infty} \sum_{k=n+1}^{\infty} \frac{n(k-n) \lvert f_n\rvert \lvert f_{k-n}\rvert}{k^2-1} \\
&= 2\sum_{n=1}^{\infty} \sum_{m=1}^{\infty} \frac{nm\lvert f_n\rvert \lvert f_m\rvert}{(n+m)^2-1} \le 4\sum_{n=1}^{\infty} \sum_{m=1}^{\infty} \frac{\sqrt{nm} \lvert f_n\rvert \lvert f_m\rvert}{n+m} \\
&\le 4\pi \left( \sum_{n=1}^{\infty} n\lvert f_n\rvert^2\right) = 4\lVert f\rVert^2_{\dot{H}^{1/2}(S^1)},
\end{align*}
where the inequality in the last line is precisely the well-known Hilbert double series theorem (\cite{HLP1934}, Section 9.1).
\end{proof}

Applying this Theorem to the terms in \eqref{Gdef}, we obtain the following straightforward Corollary which takes care of the second term in the equation \eqref{forcing} for $u_{\theta}$ in the Euler-Weil-Petersson equation.

\begin{corollary}\label{Gboundcor}
Suppose $u$ is vector field on $S^1$, and let $G = H(1+\partial_{\theta}^2)^{-1}[2u_{\theta}Hu_{\theta} - u_{\theta\theta}Hu_{\theta\theta}]$ as in \eqref{Gdef}. Then we have
\begin{equation}
\lVert G\rVert_{L^{\infty}} \le 8\pi \lVert u\rVert^2_{\dot{H}^{1/2}(S^1)} + 4\pi \lVert u\rVert^2_{\dot{H}^{3/2}(S^1)},
\end{equation}
in terms of the degenerate seminorm $\lVert u\rVert^2_{\dot{H}^{3/2}(S^1)} = \int_{S^1} (Hu)(u'''+u') \, d\theta$.
\end{corollary}

\subsection{Proof of Theorem \ref{mainthm2}}

The work of Escher and Kolev shows that solutions of \eqref{EWP} are global as long as we can control the $C^1$ norm $\lVert u\rVert_{C^1(S^1)}$. This follows from the no-loss/no-gain Lemma 4.1 of \cite{EK2014a} and the general estimate for Sobolev $H^q$ norms in terms of $C^1$ norms from Theorem 5.1 of \cite{EK2014b}.

\begin{theorem}[Escher-Kolev]\label{C1boundthm}
Let $u$ be a smooth solution of \eqref{EWP} on a maximal time interval $[0,T)$. If there is a constant $C$ such that $\lVert \partial_{\theta} u(t,\theta)\rVert_{L^{\infty}(S^1)} \le C(1+t)$ for all $t\in [0,T)$, then in fact $T=\infty$.
\end{theorem}

Hence all we need to do is obtain a uniform bound for the $C^1$ norm of $u$. Since the $\dot{H}^{3/2}$ seminorm of a solution of \eqref{EWP} is constant by energy conservation, it is sufficient to bound the $C^1$ norm in terms of the $\dot{H}^{3/2}$ seminorm. Note that the $H^{3/2}(S^1)$ norm does not in general control the $C^1(S^1)$ norm of an arbitrary function $f$ on $S^1$; we need to use the special structure of the equation \eqref{EWP} to get this.

\begin{proof}[Proof of Theorem \ref{mainthm2}]
Proposition \ref{EWPrewrite} shows that
$$ \frac{d}{dt} \lVert u_{\theta}\rVert_{L^{\infty}} \le \lVert F\rVert_{L^{\infty}} + \lVert G\rVert_{L^{\infty}}.$$
Using Theorem \ref{mysteryboundthm}, we obtain $\lVert F\rVert_{L^{\infty}} \le C \lVert u_{\theta}^2\rVert_{L^2} + C E_0$ where
$E_0 = \lVert u\rVert^2_{\dot{H}^{3/2}}$, which is constant in time since $u$ is an Euler-Arnold equation. Similarly Corollary \ref{Gboundcor} yields $\lVert G\rVert_{L^{\infty}} \le C \lVert u_{\theta}\rVert^2_{\dot{H}^{1/2}} + C E_0$. Since $u$ is always chosen as the representative of the equivalence class that has $c_1=c_{-1}=0$ (i.e., its Fourier coefficients are only nonzero for $\lvert n\rvert \ge 2$), we can easily bound both of these lower-order terms above by some constant multiple of $E_0$.

It follows that $\frac{d}{dt} \lVert u_{\theta}\rVert_{L^{\infty}} \le C E_0$, so that $\lVert \partial_{\theta}u(t)\rVert_{L^{\infty}} \le \lVert \partial_{\theta} u_0\rVert_{L^{\infty}} + C E_0 t$, which shows that $u_{\theta}$ cannot approach infinity in finite time. This proves that the solution of the EWP equation \eqref{EWP} remains in any $H^s$ space that $u_0$ begins in for any $s>\frac{3}{2}$, using Theorem \ref{C1boundthm}.
\end{proof}

\subsection{Conformal Welding on $T(1)$ and $\mathcal{T}(1)$}
In Teichm\"uller theory there is a natural identification between diffeomorphisms of the unit circle and curves in the plane, up to certain normalizations, given by conformal welding. Here we use numerical simulations to map the geodesics  corresponding to the Wunsch and EWP equations to their respective spaces of curves in $\mathbb{C}$, which then correspond to spaces of appropriately normalized conformal maps of the unit disk. It would be interesting to use this alternative representation of geodesics to prove the results of Theorem \ref{mainthm1} directly, since it could generate
a new approach to other Euler-Arnold equations with similar geometric properties
for which these results are not known, such as the surface quasigeostrophic
equation~\cite{Was2015} and the 3D Euler equation~\cite{PS2015}. We really have two different situations where we may apply conformal welding: for the universal Teichm\"{u}ller space $T(1)$, and for the universal Teichm\"{u}ller curve $\mathcal{T}(1)$.

\subsubsection{The universal Teichm\"uller space $T(1)$}
As discussed in Sharon and Mumford \cite{SM2006}, for the universal Teichm\"{u}ller space $T(1)$, conformal welding may be thought of as a map
$$\Gamma:\mathcal{C}\rightarrow \Diff(S^1)/\PSL_2(\mathbb{R}),$$
where $\mathcal{C}$ consists of equivalence classes of smooth, closed, simple curves in $\mathbb{C}$ modulo scaling and translation operations. Given a representative $C$ of an equivalence class $[C]\in \mathcal{C}$, by the Riemann Mapping Theorem, there exists a univalent, holomorphic function $\Phi_-:\mathbb{D}\rightarrow \mathbb{C}$ such that $\partial \Phi_-(\mathbb{D})=C$. This function is unique up to precomposition by an element of $\PSL_2(\mathbb{R})$.  Then, we also find a (this time unique) holomorphic map $\Phi_+:\mathbb{D}^*\rightarrow \mathbb{C}$, where $\mathbb{D}^*$ is the exterior of the unit disk, such that $\partial \Phi_+(\mathbb{D}^*)=C$, $\Phi_+(\infty)=\infty$ and $\Phi_+'(\infty)>0$.

We then define
$$\Gamma: [C] \mapsto [\eta]_{\PSL_2(\mathbb{R})}\ni \eta=\Phi_+^{-1}\circ \Phi_-|_{S^1}$$
The fact that this map is bijective follows from the construction of its inverse, see Sharon and Mumford \cite{SM2006}, or Lehto \cite{Leh1987}. One may construct the inverse by solving a Fredholm integral equation of the second kind as in Feiszli and Mumford \cite{FM2007}. One must solve:
\begin{equation}\label{weldkern}
K(F)+F=e^{i\theta}
\end{equation}
where $F(\theta)=\Phi_+(e^{i\theta})$ and
$$K(F)(\theta)=\frac{i}{2}\int_{S^1}\left(\cot\left(\frac{\theta-\psi}{2}\right)-(\eta^{-1})'(\psi)\cot\left(\frac{\eta^{-1}(\theta)-\eta^{-1}(\psi)}{2}\right)\right)F(\psi)\,d\psi.$$
The results of Feiszli and Mumford \cite{FM2007} ensure that \eqref{weldkern} has a unique solution. Thus we have a practical way of computing
$$\Gamma^{-1}:\eta \mapsto F([0,2\pi)).$$

\subsubsection{The universal Teichm\"uller curve $\mathcal{T}(1)$}
Alternatively, as discussed in Kirillov \cite{Kir1987} and Teo \cite{Teo2004}, one may consider conformal welding for the universal Teichm\"{u}ller curve, $\mathcal{T}(1)$, which may be thought of as a map:
$$\tilde{\Gamma}:\tilde{\mathcal{C}}\rightarrow \Diff(S^1)/\Rot(S^1),$$
where $\tilde{\mathcal{C}}$ consists of smooth, closed, simple curves in $\mathbb{C}$ with conformal radius one at the origin. We then proceed in a similar fashion as before. This time, when we find a univalent, holomorphic $\Phi_-:\mathbb{D}\rightarrow \mathbb{C}$ such that $\partial \Phi_-(\mathbb{D})=C$, we also demand that $\Phi_-$ be the unique mapping such that $\Phi_-(0)=0$ and $\Phi_-'(0)>0$ (which is the normalization demanded by the Riemann Mapping Theorem). Since $C$ has conformal radius one, $\Phi_-'(0)=1$. We find $\Phi_+$ in the same way as before, and then have the same definition for $\tilde{\Gamma}$ as for $\Gamma$:
$$\tilde{\Gamma}:C\mapsto [\eta]_{\Rot(S^1)}\ni \eta = \Phi_+^{-1}\circ \Phi_-|_{S^1}$$
where the equivalence class is now taken to be in $\Diff(S^1)/\Rot(S^1)$ as opposed to $\Diff(S^1)/\PSL_2(\mathbb{R})$ in the case of $\Gamma$. That this map is bijective is proved in Kirillov \cite{Kir1987} for smooth curves and in Teo \cite{Teo2004} for the more general situation of quasicircles. As mentioned above, our goal was to use numerical simulations to take the Lagrangian trajectories corresponding to the Wunsch equation, and map them (at each time $t$) via $\tilde{\Gamma}^{-1}$, to $\tilde{\mathcal{C}}$. We employed the integral kernel method from Feiszli and Mumford \cite{FM2007}. At the moment, our $[\eta]_{\Rot(S^1)}$ will be mapped via this method to the correct curve in $\mathbb{C}$ up to scaling and translation.  The curve can be readily normalized so that it encloses 0 and has conformal radius 1 (i.e. so that it lies within $\tilde{\mathcal{C}}$). Of course, the issue is that there is a whole family of possible such normalized curves within an equivalence class in $\mathcal{C}$.\\
\\
This issue can be readily solved by observing the following. Suppose that $C_1, C_2\in \tilde{C}$ such that $a(C_1+v) = C_2$ for some $v\in \mathbb{C}$, and $a\in \mathbb{R}^+$. Let $f_1:\mathbb{D}\rightarrow \mathbb{C}$ and $f_2:\mathbb{D}\rightarrow \mathbb{C}$ be the corrseponding Riemann mappings with boundary values $C_1$ and $C_2$ respectively, with $f_i(0)=0$ and $f_i'(0)=1$. Then there is a unique element $\phi\in \PSL_2(\mathbb{R})$ such that $ a(f_1\circ \phi +v) = f_2$. This explains the differences in normalizations between the $T(1)$ and $\mathcal{T}(1)$ situations and now motivates the technique for our numerical simulation computing $\tilde{\Gamma}^{-1}$. First, we take our geodesic $\eta(t)$ corresponding to the Wunsch equation in Lagrangian coordinates. We employ the algorithm in Feiszli and Mumford \cite{FM2007} and solve \eqref{weldkern} to obtain the curve we want up to scaling and translation. We take any normalization (up to translation and scaling) of this curve so that it lies in $\tilde{C}$. We then find $\Phi_+$ and $\Phi_-$ and compute $\zeta= (\Phi_+)^{-1}\circ \Phi_-|_{S^1}$. By previous arguments we have $\zeta\in [\eta]_{\PSL_2(\mathbb{R})}$. Hence there exists a $\phi \in \PSL_2(\mathbb{R})/\Rot(S^1)$ such that $\zeta \circ \phi|_{S^1} =\eta$ and we compute $\phi|_{S^1}=\zeta^{-1}\circ \eta$. Then we extend $\phi|_{S^1}$ to $\mathbb{D}$ (simply by observing that we only need to know 3 values of $\phi$ to know $\phi$ itself), and then we compute $\phi(0)=(f_1)^{-1}(-v)$. This yields the correct curve in $\tilde{C}$.

\subsection{Numerical Simulations}

In this section we show the results of numerical simulations solving the Wunsch and Euler-Weil-Petersson equations as well as implementing the conformal welding process above.
%along with more specific documentation on how to use the code.
%The suite has the capability to:
%
%\begin{itemize}
% \item Solve the Wunsch and Euler-Weil-Petersson equations in Eulerian coordinates.
%
% \item Solve the flow equation to switch to Lagrangian coordinates
%
% \item Plot the Ermakov Pinney trajectories of each equation
%
% \item Map the geodesics to the space of univalent holomorphic functions with the appropriate normalizations
% ($\mathcal{D}$ or $\tilde{\mathcal{D}}$) via conformal welding.
%
% \item Undo the conformal welding map via the Schwarz-Christoffel formula (as is done in Sharon and Mumford \cite{SM2006})
%
%\end{itemize}
\subsubsection{Solutions to EWP and Wunsch}
Here we implemented a Fourier-Galerkin method to get a system of ODES, coupled with a 4th order Runge-Kutta method to solve each ODE that arises. The following is a collection of solutions for the EWP and Wunsch equations with initial condition $u_0(x) = \sin(2x)$. For each equation we have $t_0=0$ and $t_{fin}=.5$.
\begin{table}[H]
\caption {Eulerian Solutions to Wunsch with $u_0=\sin(2x)+\frac{1}{2}\cos(3x)$. Note that the slopes approach $-\infty$; after this the numerical solution appears to become singular everywhere simultaneously. It is not clear if this is what actually happens.}
\begin{tabular}{ccc}
\includegraphics[height=1.6in]{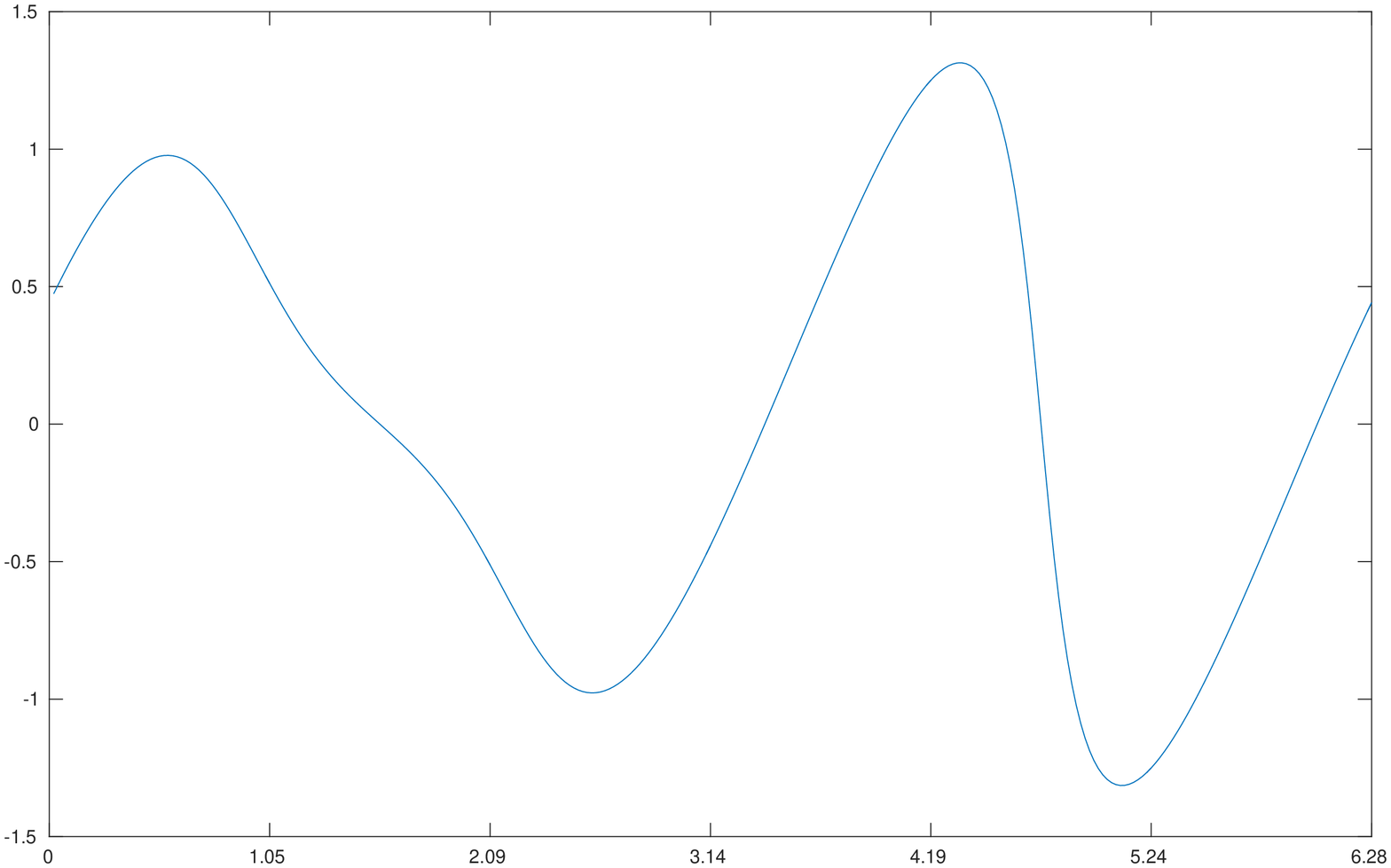} &
\includegraphics[height=1.6in]{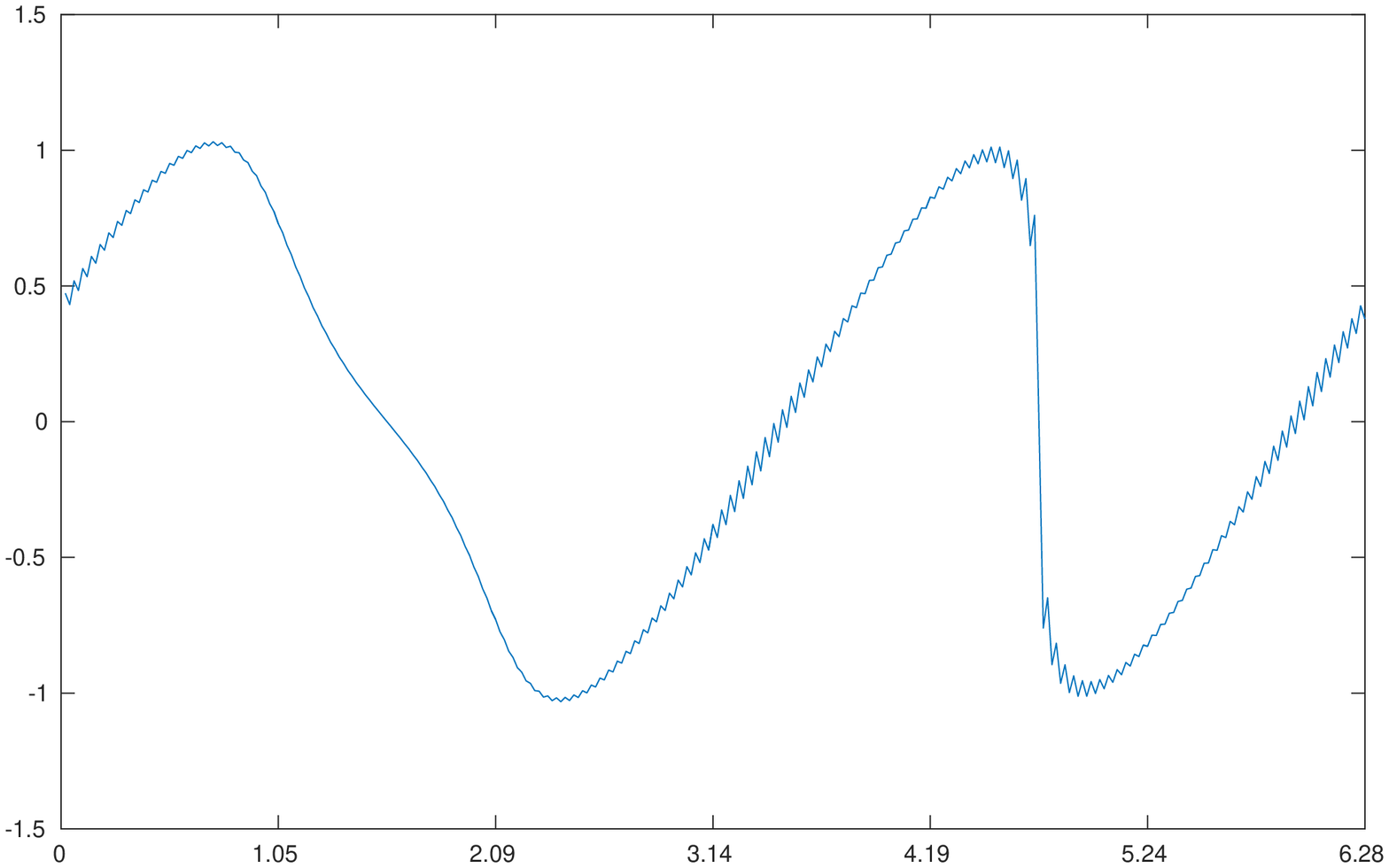} \\
t=.125 (before blowup) & t=.25 (after blowup) \\

\end{tabular}
\end{table}

\begin{table}[H]
\caption {Eulerian Solutions to EWP with $u_0=\sin(2x)+\frac{1}{2}\cos(3x)$. The profile steepens but does not become singular.}
\begin{tabular}{cc}
\includegraphics[height=1.6in]{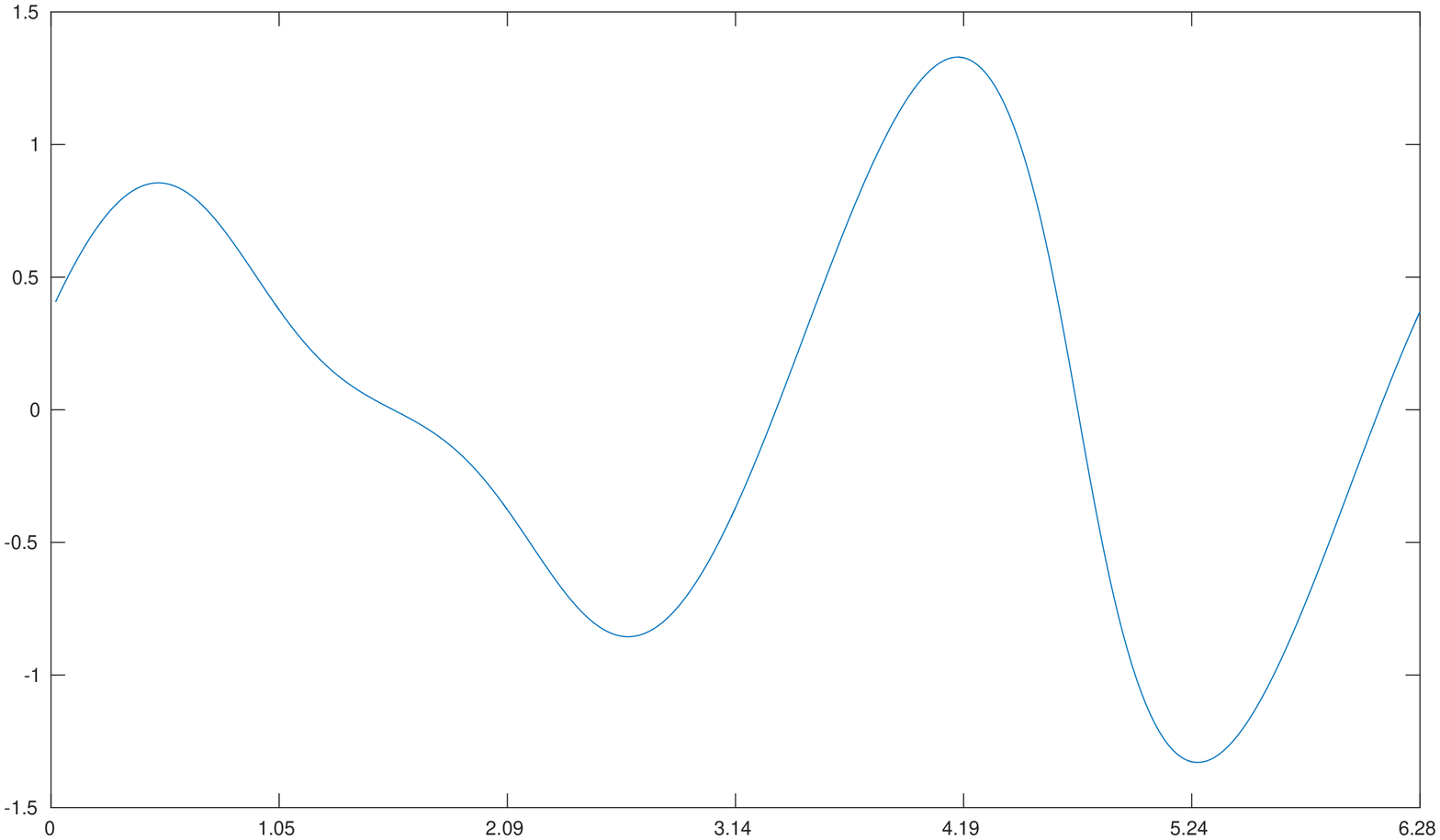} &
\includegraphics[height=1.6in]{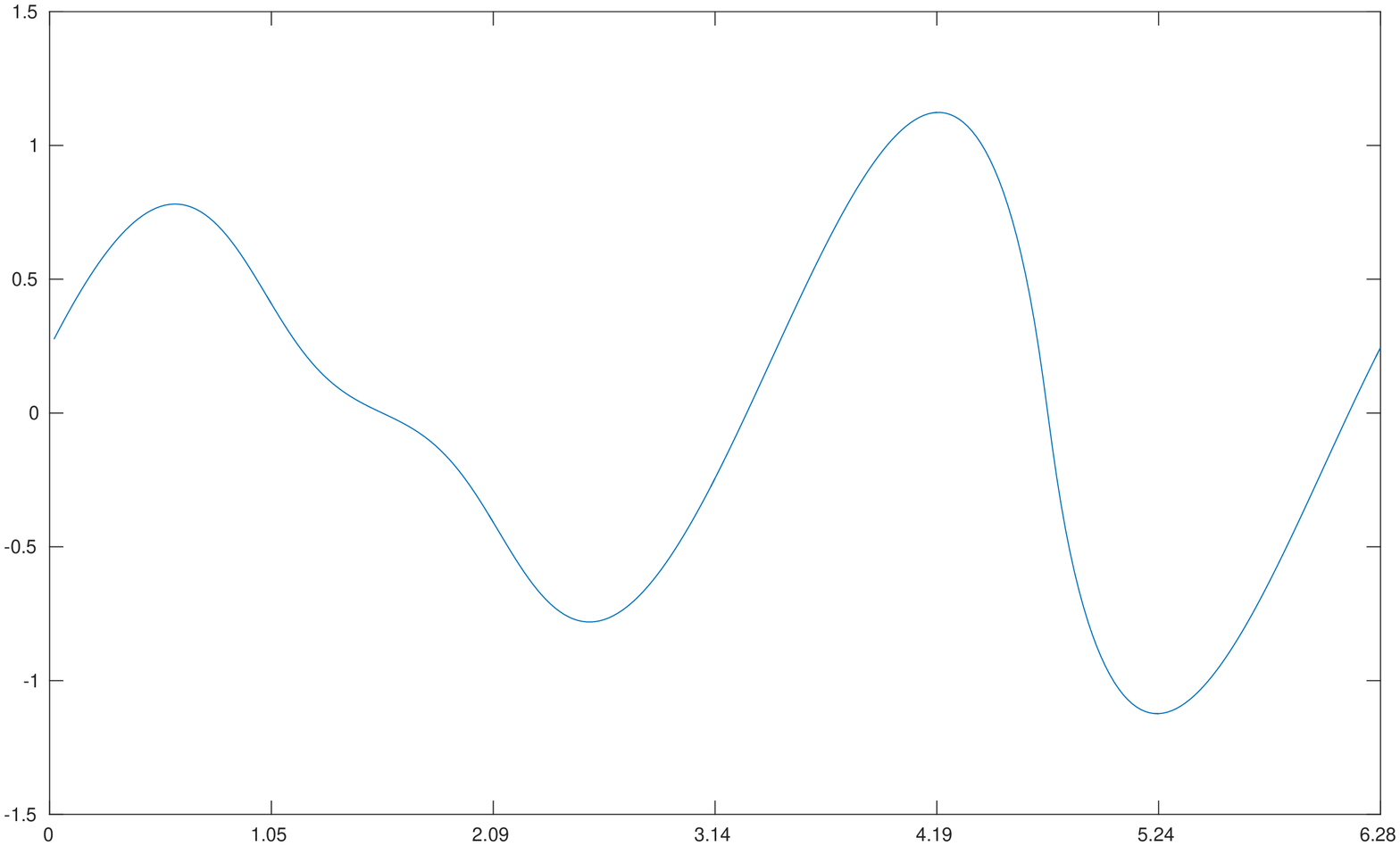} \\
t=.25  & t=.5\\

\end{tabular}
\end{table}

\begin{table}[H]
\caption {Lagrangian Solutions to Wunsch with $u_0=\sin(2x)+\frac{1}{2}\cos(3x)$. As $u_{\theta}$ approaches $-\infty$, the slope of $\eta$ approaches zero, and $\eta$ leaves the diffeomorphism group.}
\begin{tabular}{cc}
\includegraphics[height=1.6in]{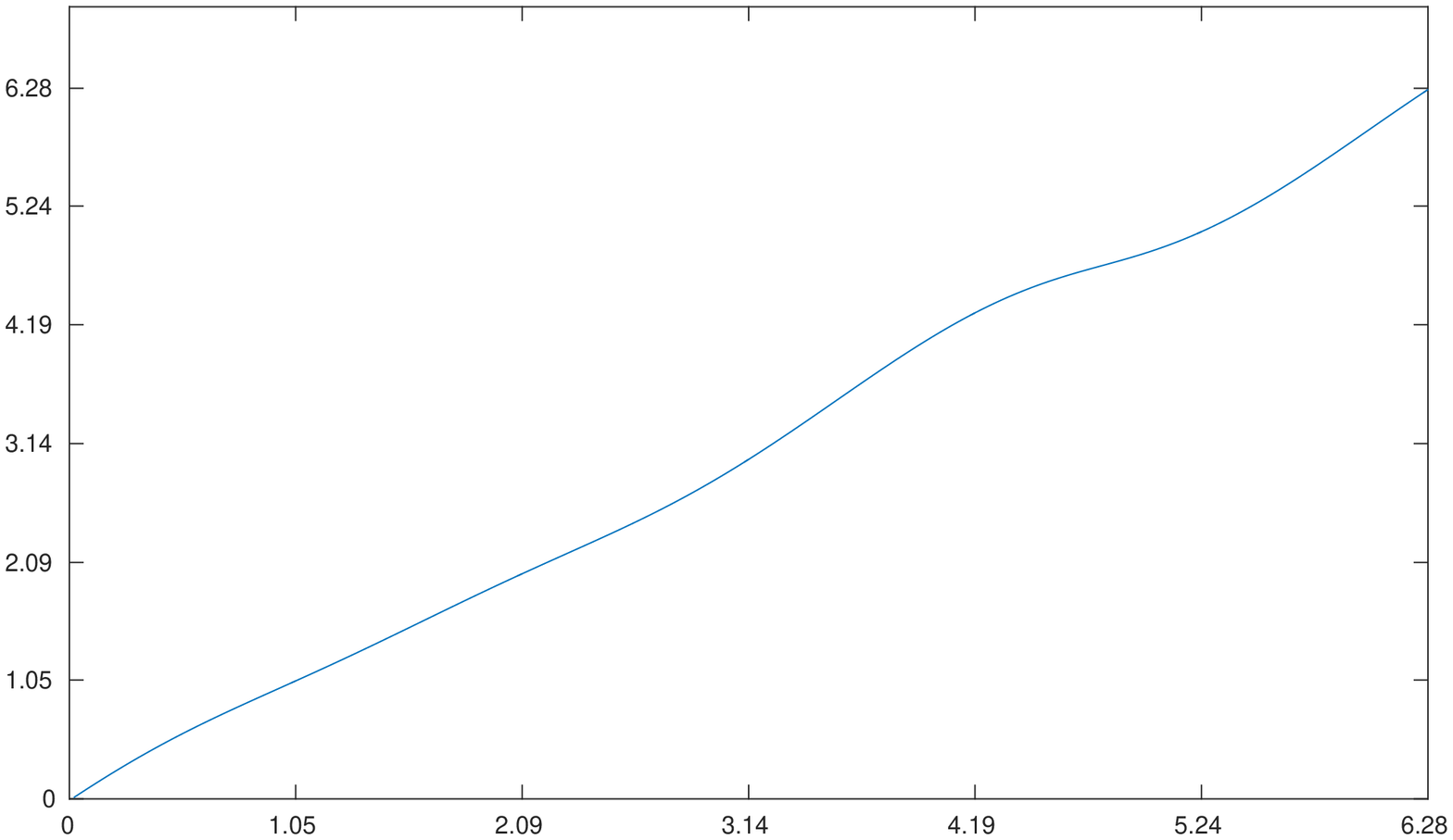} &
\includegraphics[height=1.6in]{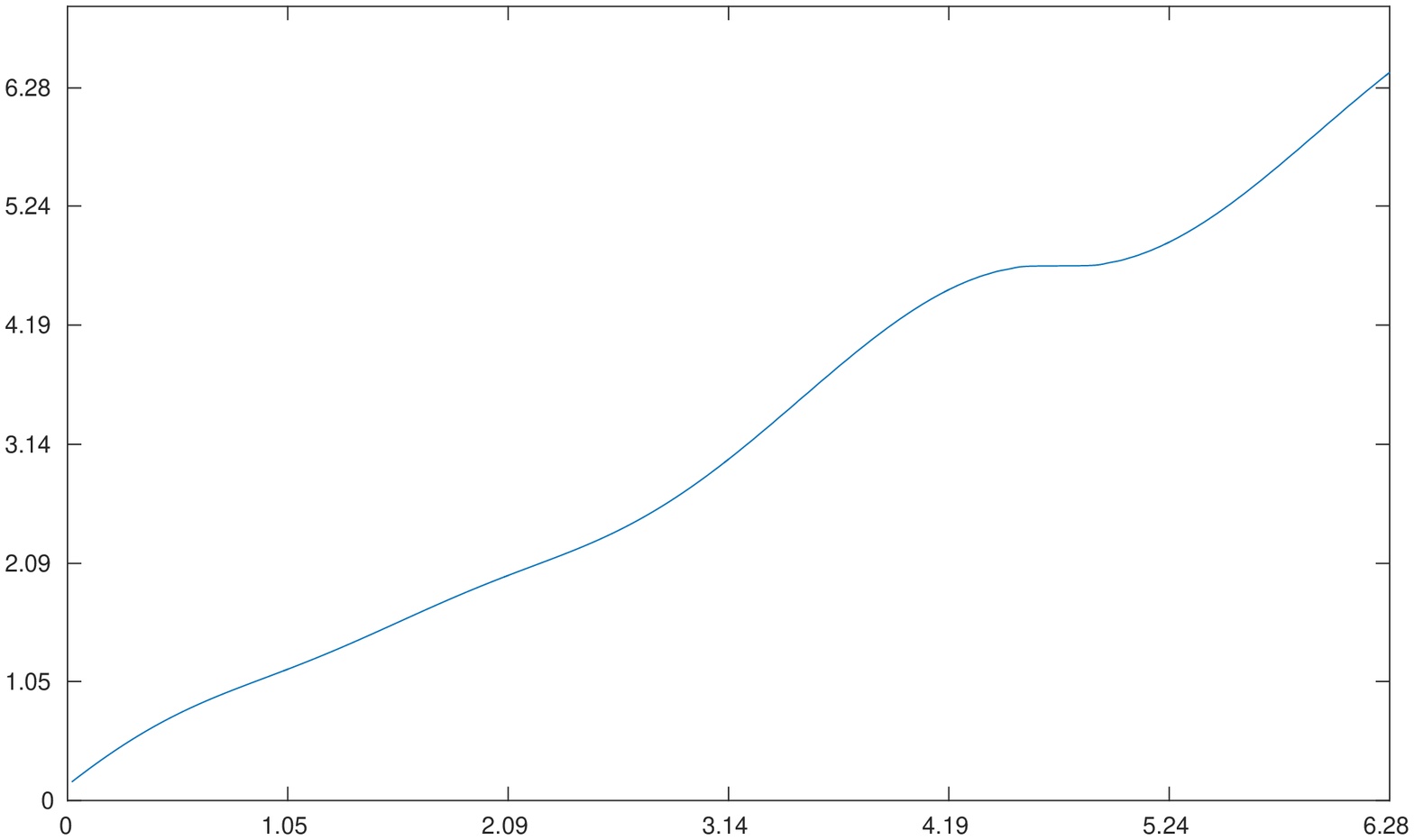} \\
t=.125 (before blowup)& t=.25 (after blowup)\\

\end{tabular}
\end{table}

\begin{table}[H]
\caption {Lagrangian Solutions to EWP with $u_0=\sin(2x)+\frac{1}{2}\cos(3x)$. It appears that $\eta$ is flattening substantially, but the slope still remains positive.}
\begin{tabular}{cc}
\includegraphics[height=1.6in]{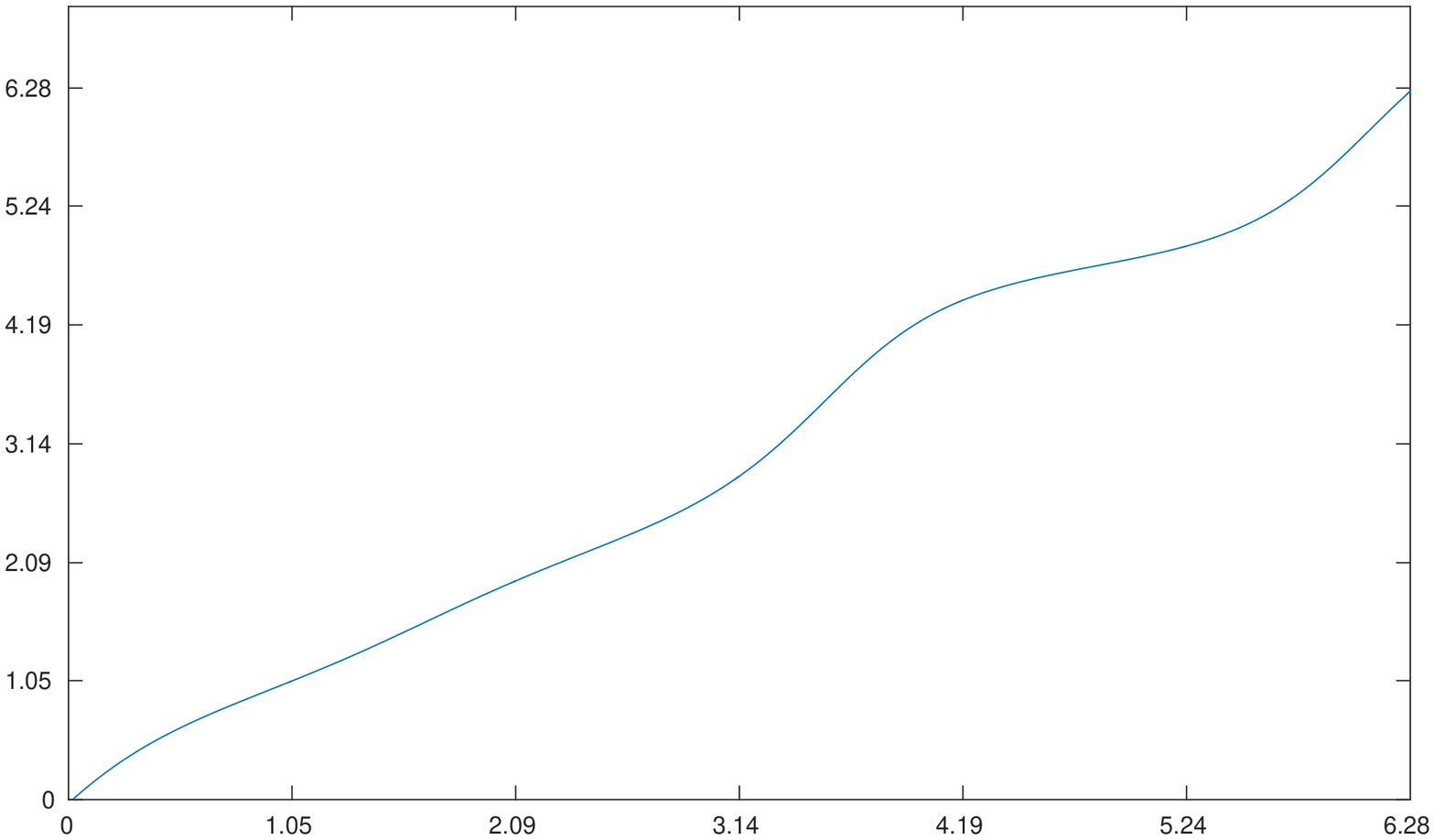} &
\includegraphics[height=1.6in]{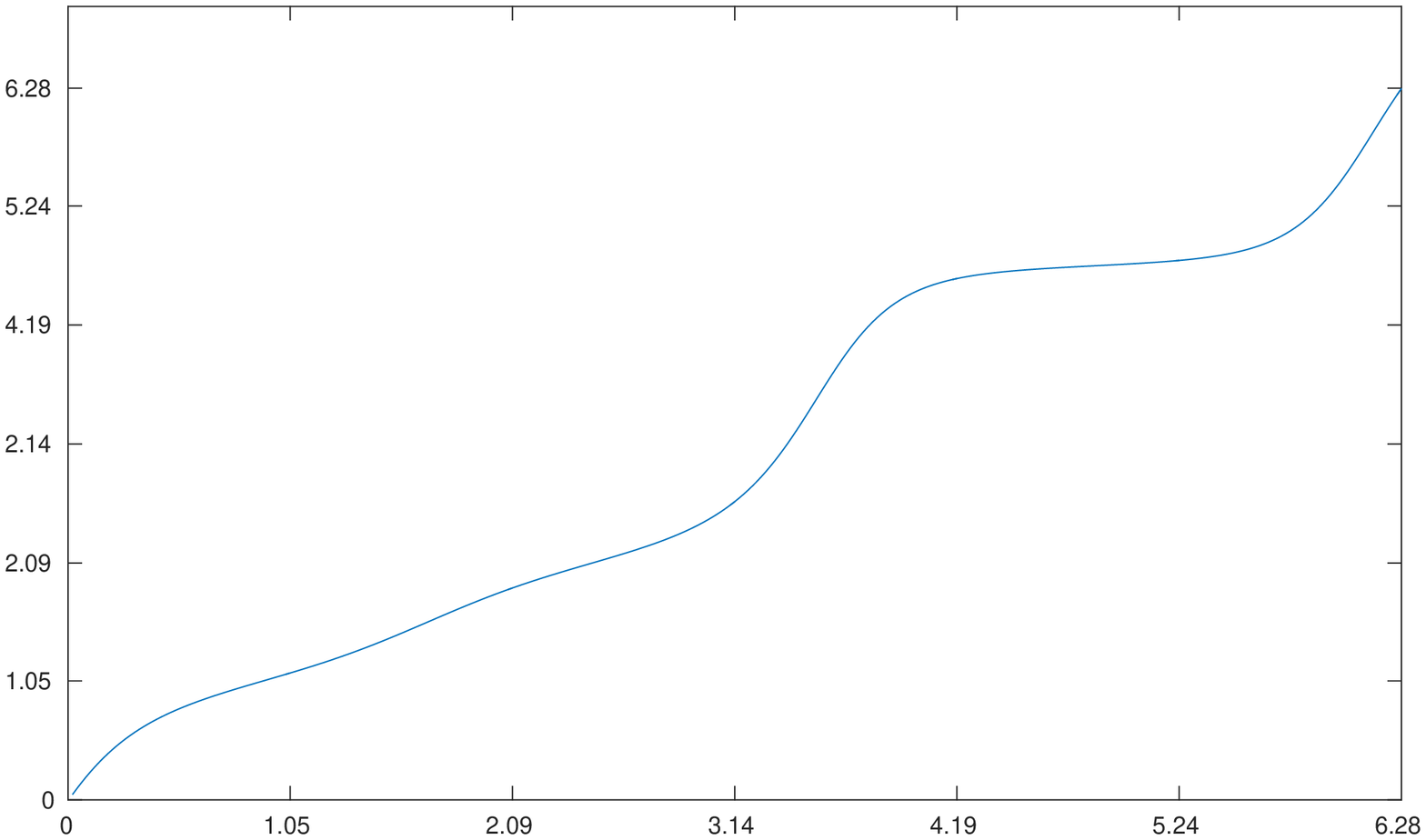} \\
t=.25  & t=.5 \\

\end{tabular}
\end{table}

\subsubsection{Conformal welding}
%\textcolor{red}{Need to contact Mumford and get his code, then cite as personal communication.}
Here we used Mumford's MATLAB code~\cite{Mum2016}
%available from his powerpoint slides at \url{https://www.math.stonybrook.edu/~ccg2007/},
to solve \eqref{weldkern}. For EWP solutions, this yields a representative of the corresponding equivalence class in $\mathcal{C}$. For solutions to the Wunsch equation, we proceed by normalizing using the technique discussed above. As in Sharon and Mumford \cite{SM2006}, we employed the Schwarz-Christoffel map, as implemented in Tobin Driscol's code available at the website \url{http://www.math.udel.edu/~driscoll/SC/index.html} to construct the Riemann mappings and hence diffeomorphisms associated to each curve. The first thing we note about the welding curves is that for the above Wunsch solutions, the translation and scaling operations as discussed above are essentially trivial. We believe that this has to do with the fact that the $\sin(x)$ Fourier mode in the above solution to the Wunsch equation is zero. This at least makes heuristic sense as the bulk of the translating and dilating should happen in the $\PSL_2(\mathbb{R})$ fibre, which in this case is zero. One can obtain small amounts of shifting and translating by making sure that the initial data has either $\sin(x)$ or $\cos(x)$ terms. Second, we note the similarity in the shapes of the solutions. Again, this is perhaps not surprising as the $H^{3/2}$ metric is obtained by averaging over the fibres of the $H^{1/2}$ metric as in Teo \cite{Teo2004}.
\begin{table}[H]
\caption {Trajectories in $\tilde{\mathcal{C}}$ after welding solutions to the Wunsch equation with $u_0=\sin(2x)+\frac{1}{2}\cos(3x)$. At the blowup time it appears that the boundary curve stops being $C^1$, though not as expected due to an easily visible kink at a point. Instead the curve appears to spontaneously generate two kinks that spread apart after the blowup time.}
\begin{tabular}{cc}
\includegraphics[height=1.6in]{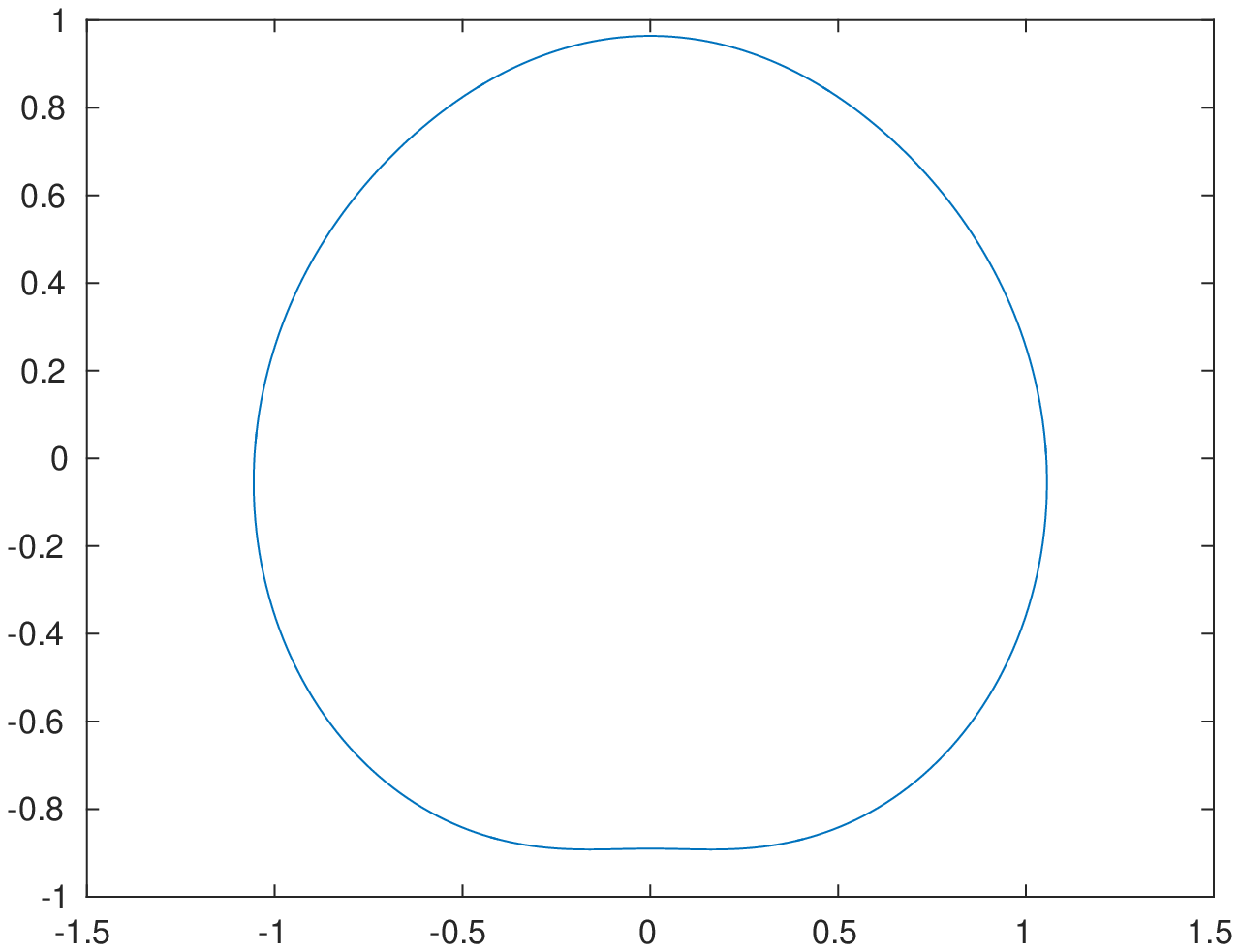} &
\includegraphics[height=1.6in]{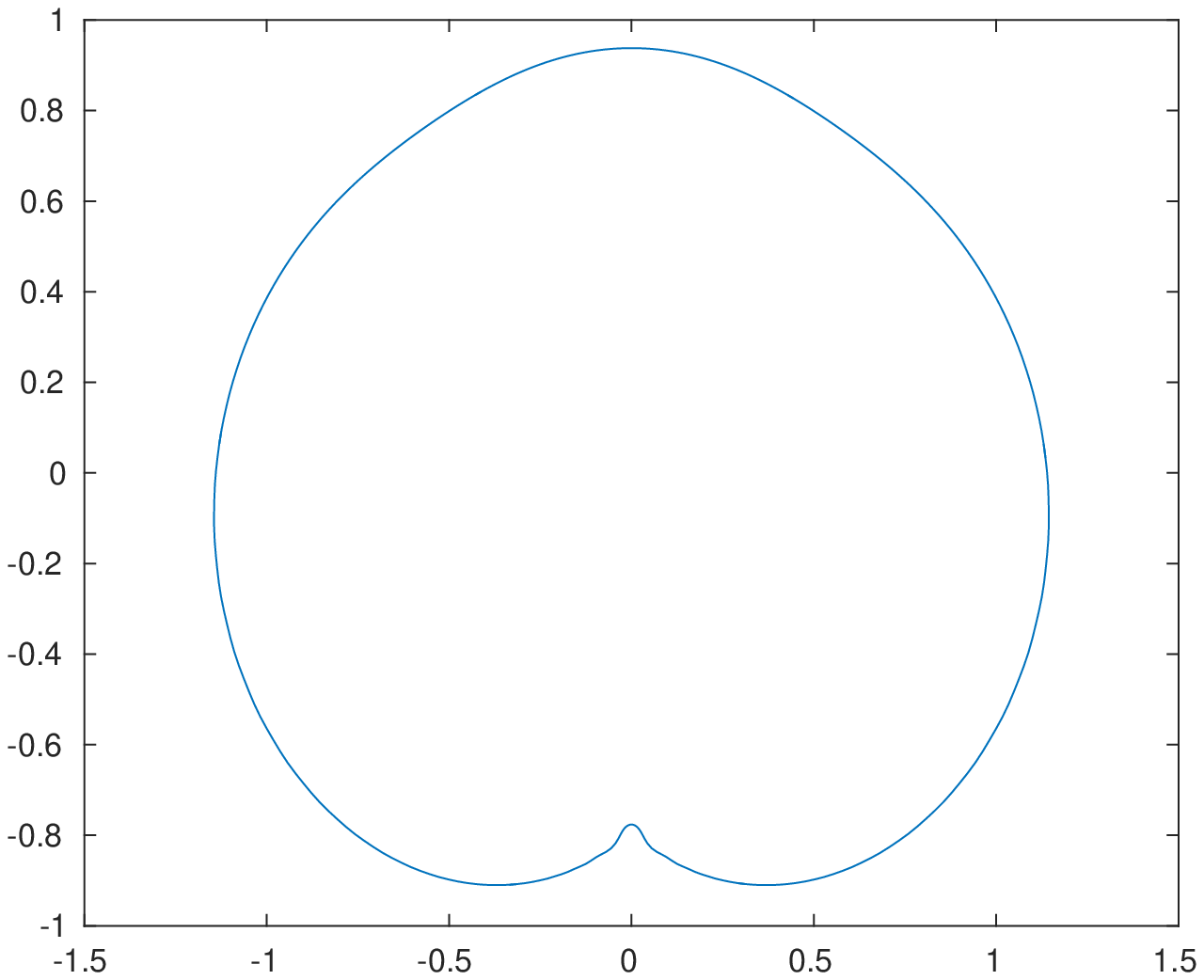} \\
t=.125  &  t=.25 (after blowup)\\

\end{tabular}
\end{table}
\begin{table}[H]
\caption {Trajectories in $\mathcal{C}$ after welding solutions to EWP with $u_0=\sin(2x)+\frac{1}{2}\cos(3x)$. The change in concavity seems to remain smooth, with continuous tangents.}
\begin{tabular}{cc}
\includegraphics[height=1.6in]{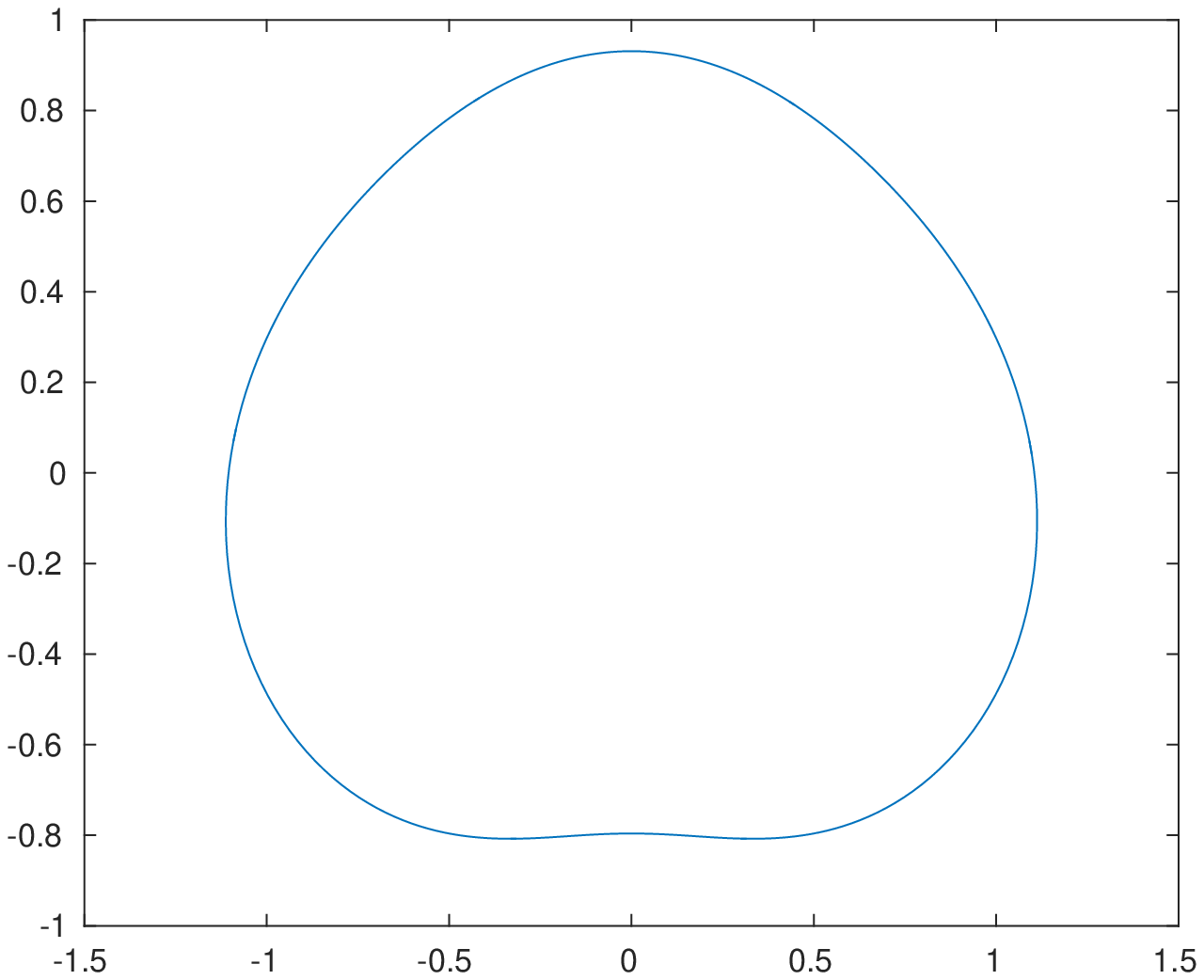} &
\includegraphics[height=1.6in]{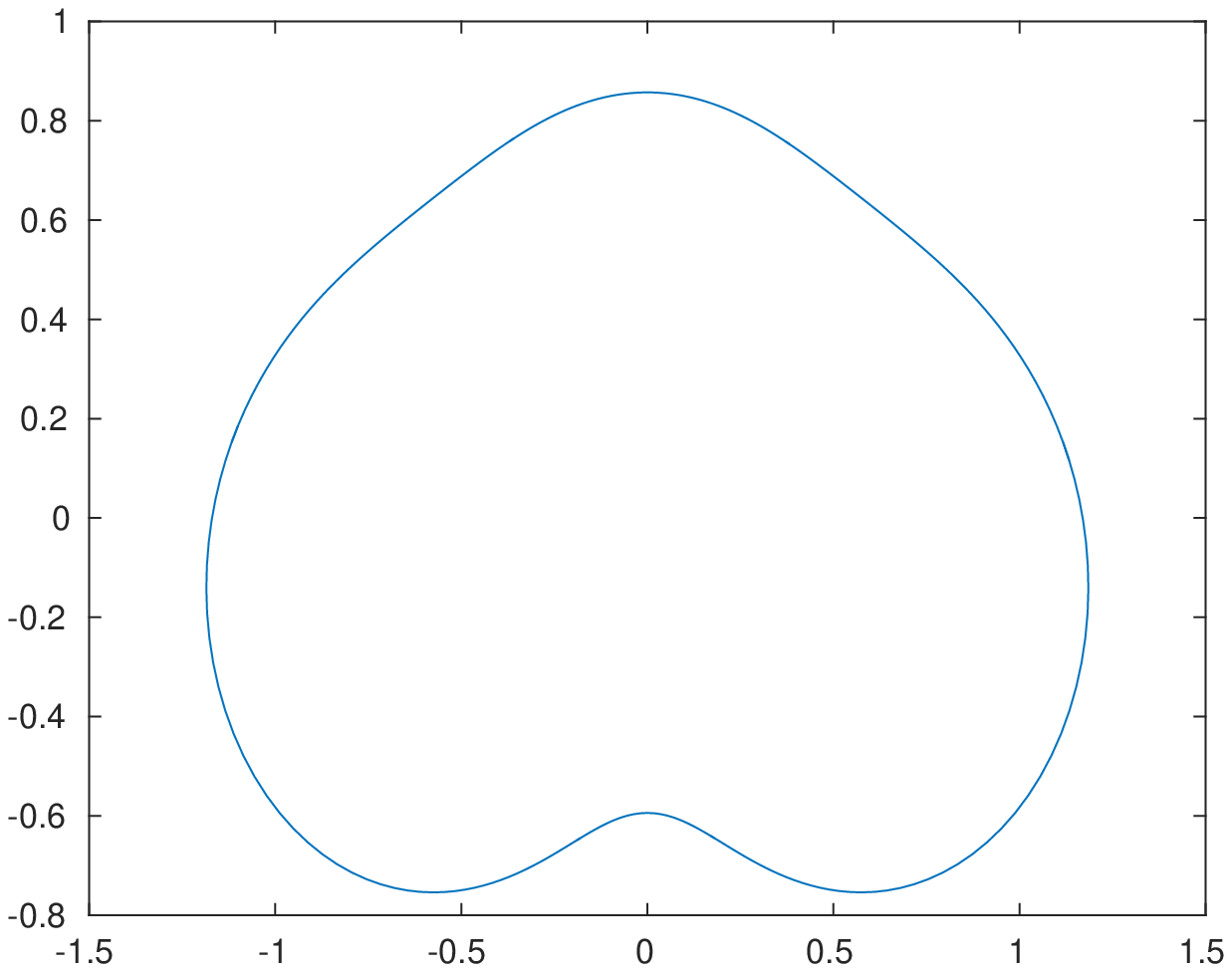} \\
t=.25  & t=.5 \\

\end{tabular}
\end{table}

%
%\subsubsection{Inverting the Welding Map}
%As in Sharon and Mumford \cite{SM2006}, we employ the Schwarz-Christoffel map, as implemented in Tobin Driscol's code available at the website \url{http://www.math.udel.edu/~driscoll/SC/index.html} to construct the Riemann mappings $\Phi_+$ and $\Phi_-$ associated to each curve. We then compute $\Phi_+^{-1} \circ \Phi_-|_{S^1}$ explicitly. There is good agreement with the original diffeomorphisms.
%\begin{table}[H]
%\caption {Recovered diffeomorphisms after welding for the Wunsch Equation}
%\begin{tabular}{ccc}
%\includegraphics[height=1.5in]{zetanormt100wunsch.eps} &
%\includegraphics[height=1.5in]{zetanormt200wunsch.eps} &
%\includegraphics[height=1.5in]{zetanormt250wunsch.eps} \\
%t=100  & t=200 &t=250\\
%
%\end{tabular}
%\end{table}
%
%\begin{table}[H]
%\caption {Recovered diffeomorphisms after welding for the EWP Equation}
%\begin{tabular}{ccc}
%\includegraphics[height=1.5in]{zetat100EWP.eps} &
%\includegraphics[height=1.5in]{zetat200EWP.eps} &
%\includegraphics[height=1.5in]{zetat300EWP.eps} \\
%t=100  & t=200 &t=300\\
%
%\end{tabular}
%\end{table}

\end{document}